\documentclass{amsart}
\usepackage{amsfonts,amssymb,amsmath,amsthm}
\usepackage{url}
\usepackage{enumerate}
\usepackage{amsmath,amsfonts,amssymb}
\newtheorem{theorem}{Theorem}[section]
\newtheorem{lemma}[theorem]{Lemma}
\newtheorem{definition}[theorem]{Definition}
\newtheorem{prop}[theorem]{Proposition}
\newtheorem{cor}[theorem]{Corollary}

\theoremstyle{definition}

\newtheorem{remark}[theorem]{Remark}
\numberwithin{equation}{section}

\font\german=eufm10
\def\a{{\hbox{\german a}}}
\def\g{{\hbox{\german g}}}

\def\k{{\hbox{\german k}}}

\def\p{{\hbox{\german p}}}
\def\so{{\hbox{\german so}}}

\def\R{{\bf R}}
\def\C{{\bf C}}
\def\F{{\bf F}}
\def\H{{\bf H}}
\def\O{{\bf O}}
\def\SO{{\bf SO}}
\def\SU{{\bf SU}}
\def\Sp{{\bf Sp}}
\def\SL{{\bf SL}}

\def\Z{\hbox{$\mathcal{Z}$}}

\def\diag{\mathop{\hbox{diag}}}
\def\Ad{\mathop{\hbox{Ad}}}
\def\ad{\mathop{\hbox{ad}}}
\author{P.\ Graczyk}
\address{
Laboratoire de Math\'ematiques\\
LAREMA\\
Universit\'e d'Angers\\
49045 Angers cedex 01}
\email{piotr.graczyk@univ-angers.fr}
\thanks{Supported by l'Agence Nationale de la Recherche
ANR-09-BLAN-0084-01}
\author{P.\ Sawyer}
\address{
Department of Mathematics and Computer Science\\
Laurentian University\\
Sudbury, Ontario}
\email{psawyer@laurentian.ca}
\begin{document}

\title[Orbital measures on spaces  of type $C_p$ and $D_p$ ]{ Convolution of orbital measures on symmetric spaces\\ of type $C_p$ and $D_p$}

\begin{abstract}
We study the absolute continuity of the convolution  $\delta_{e^X}^\natural \star\delta_{e^Y}^\natural$
of two orbital measures on  the symmetric spaces ${\bf SO}_0(p,p)/{\bf SO}(p)\times{\bf SO}(p)$,
$\SU(p,p)/{\bf S}({\bf U}(p)\times{\bf U}(p))$ and $\Sp(p,p)/{\bf Sp }(p)\times\Sp(p)$.
We prove sharp conditions on $X$, $Y\in\a$ for the existence of the density of the convolution measure. This measure
intervenes in the product formula for the spherical functions.

\end{abstract}

\maketitle

\section{Introduction} 

The spaces  $G/K={\bf SO}_0(p,p)/{\bf
SO}(p)\times{\bf SO}(p)$ are Riemannian symmetric spaces of non-compact type corresponding
to root systems of type $D_p$.  The spaces  $\SU(p,p)/{\bf S}({\bf U}(p)\times{\bf U}(p))$ and 
$\Sp(p,p)/\Sp(p)\times\Sp(p)$ correspond to root systems of type $C_p$. 

Consider  $X$, $Y\in \a$ and let $m_K$ denote the Haar measure of the group $K$. We define
$\delta_{e^X}^\natural=m_K\star \delta_{e^X} \star m_K$. It is the uniform measure on the orbit
$Ke^XK$. The problem of the absolute continuity of the convolution
\begin{align*}
m_{X,Y}=\delta_{e^X}^\natural \star\delta_{e^Y}^\natural
\end{align*}
of two orbital measures that we address in our paper
has  important applications in harmonic analysis of spherical functions  on $G/K$ and  in probability theory.
Let   $\lambda$ be a complex-valued linear form on $\a$ and
$\phi_\lambda (e^X)$ be   the spherical function,  which is 
the spherical Fourier transform of the measure $\delta_{e^X}^\natural$.
The product formula for the spherical functions states that
\begin{align*}
\phi_\lambda (e^X)\,\phi_\lambda (e^Y) =\int_{\a}\,\phi_\lambda (e^H)\,d\mu_{X,Y} (H)
\end{align*}
where $\mu_{X,Y}$ is the projection of the measure $m_{X,Y}$ on $\a$ via the Cartan decomposition $G=KAK$.
Then the existence of a density of $\mu_{X,Y}$, equivalent to the absolute continuity of $m_{X,Y}$,
is of great importance.

It was proven in \cite{PGPS2} that as soon as the space $G/K$ is irreducible and one of the elements $X$, $Y$ is regular and the 
other nonzero, then the convolution  $\delta_{e^X}^\natural \star\delta_{e^Y}^\natural$ has a density. The density can however still exist
when both $X$ and $Y$ are singular.  It is a challenging problem to characterize all such pairs $X$, $Y$. 

This problem was solved  for symmetric spaces   of type $A_n$ and for the exceptional space $\SL(3,\O)/\SU(3,\O)$ of type $E_6$ in 
\cite{PGPS_Lie2010},   and  for symmetric spaces of type $B_p$ and $BC_p$ in \cite{PGPS_2013}.  
In the present paper, we present the solution of the problem for Riemannian symmetric spaces of type $C_p$ and $D_p$. 

For the good comprehension of the methods of this paper, it is useful to know the paper \cite{PGPS_2013}. 
However, the cases $C_p$ and $D_p$ require many original ideas that did not appear in the case $B_p$.
We refer to S. Helgason's books \cite{Helga0} and \cite{Helga} for the standard notation
and results. 

In Section \ref{root}, we are reminding the reader of the basic information about the Lie group $\SO(p,p)$ and its Lie algebra $\so(p,p)$.  We also provide the necessary notation to describe the configuration of an element of the Cartan subalgebra $\a$ of $\so(p,p)$.  This configuration notion allows us to ``measure'' how singular an element of $\a$ is and to describe in a precise manner which pairs $X$, $Y\in\a$ are ``eligible'', the sharp criterion that we establish in the paper for the absolute continuity of $\mu_{X,Y}$. The following theorem is the main result of the paper:\\[1mm]
{\bf Theorem A.} {\it Let $G/K={\bf SO}_0(p,p)/{\bf
SO}(p)\times{\bf SO}(p)$ and $X$, $Y\in\a$. The density of the convolution $\delta_{e^X}^\natural \star
\delta_{e^Y}^\natural$ exists if and only if $X$ and $Y$ are eligible (see Definition  \ref{defEligible}). 
}\\

In the following Section \ref{tools}, a series of definitions and accessory results are given to set the stage for the proof of Theorem A.
In Section \ref{NSSE}, we show that $(X,Y)$ has to be an  eligible pair for the measure $\mu_{X,Y}$ to be absolutely continuous. 
In Section \ref{Suff}, we then show that the eligibility condition is also sufficient.  

In the last section, as in \cite{PGPS_Lie2010} and \cite{PGPS_2013}, we extend our results to the complex and quaternionic cases.  
Again, the richness of the root structure comes into play: in the table in Remark \ref{complex}, we find that the 
complex and quaternionic cases have much more in common with the cases $q>p$ than with the real case $\SO(p,p)$.

We conclude the paper with a discussion of the absolute continuity of powers of $\delta_{e^X}^\natural$ for a  nonzero $X\in\a$.
%%%%%%%%%%%%%%%%%%%%%%%%%%%%%%%%%%%%%%%%%%%%%%%%%%%%%%%%%
\section{Preliminaries and definitions}\label{root}
%%%%%%%%%%%%%%%%%%%%%%%%%%%%%%%%%%%%%%%%%%%%%%%%%%%%%%%%%%%%%
We start by reviewing some useful information on the Lie group
$\SO_0(p,p)$, its Lie algebra $\so(p,p)$ and the corresponding
root system. Most of this material was given in \cite{Sawyer1}. 
For the convenience of the reader, we gather below the properties
we will need in the sequel.  In this paper, $E_{ij}$ is a rectangular matrix with 0's everywhere except 
at the position $(i,j)$ where it is 1.
Recall that $\SO(p,p)$ is the group of matrices $g\in{\bf SL}(2\,p,\R)$ such that $g^T\,I_{p,p}\,g=I_{p,p}$
where $I_{p,p}= \left [
\begin{array}{cc}
-I_p&0_{p\times p}\\ 
0_{p\times p}&I_p 
\end{array}
\right]$.  Unless otherwise specified, all $2\times 2$ block decompositions in this paper follow the same pattern.
The group $\SO_0(p,p)$ is the connected component of $\SO(p,p)$ containing the identity.
The Lie algebra $\so(p,p)$ of $\SO_0(p,p)$ consists of the matrices 
\begin{align*}
\left [\begin{array}{cc}A&B\\ B^T&D
\end{array}\right]
\end{align*}
where $A$ and $D$ are skew-symmetric. A very important element in our investigations is the Cartan
decomposition of $\so(p,p)$ and $\SO(p,p)$. The maximal compact subgroup $K$ is the subgroup of $\SO(p,p)$ 
consisting of the matrices
\begin{align*}
\left [\begin{array}{rr} 
A&0\\ 
0&D \end{array}\right]
\end{align*}
of size $(2\,p)\times (2\,p)$ such that $A$, $D\in \SO(p)$ (hence $K \simeq \SO(p)
\times \SO(p)$). If $\k$ is the Lie algebra of $K$ and $\p$ is the set of matrices 
\begin{align}\label{pp}
\left [\begin{array}{cc} 
0&B\\ 
B^T&0 \end{array}\right]
\end{align}
then the Cartan decomposition is given by $\so(p,p)=\k\oplus\p$ with corresponding Cartan involution
$\theta(X)=-X^T$. To shorten the notation, for $X\in\p$ as in (\ref{pp}), we will write $X^s=B$.\\
The Cartan space $\a\subset \p$ is the set of matrices 
\begin{align*}H=\left
[\begin{array}{cc}
0_{p\times p}&\mathcal{D}_H\\ 
\mathcal{D}_H&0_{p\times p}
\end{array}\right]
\end{align*}
where $\mathcal{D}_H=\diag[H_1,\dots,H_p]$. Its canonical basis is given by the matrices
\begin{align*}
A_i:=E_{i,p+i} + E_{p+i,i}, 1\leq i\leq p.
\end{align*} 
The restricted roots and associated root vectors for the Lie algebra
$\so(p,p)$ with respect to $\a$ are given in Table \ref{X}.
\begin{table}[h]
\begin{center}
\begin{tabular}{|c|c|c|}\hline
root $\alpha$&multiplicity&root vectors $X_\alpha$\\ \hline
$\alpha(H)=\pm(H_i-H_j)$&1&
$Y_{i,j}^\pm=\pm(E_{i,j}-E_{j,i}+E_{p+i,p+j}-E_{p+j,p+i})
+E_{i,p+j}+E_{p+j,i}$\\
$1\leq i,j\leq p$, $i< j$&&${}+E_{j,p+i}+E_{p+i,j}$\\ \hline
$\alpha(H)=\pm(H_i+H_j)$&1&
$Z_{i,j}^\pm=\pm(E_{i,j}-E_{j,i}-E_{p+i,p+j}+E_{p+j,p+i})
-(E_{i,p+j}+E_{p+j,i})$\\ 
$1\leq i,j\leq p$, $i<j$&&${}+E_{j,p+i}+E_{p+i,j}$\\ \hline
\end{tabular}
\end{center}
\caption{Restricted roots and associated root vectors}\label{X}
\end{table}
The positive roots can be chosen as $\alpha(H)=H_i\pm H_j$, $1\leq i<j\leq p$. 
The simple roots are given by $\alpha_i(H)=H_i-H_{i+1}$, $i=1$, \dots, $p-1$ and $\alpha_p(H)=H_{p-1}+H_p$.
We therefore have the positive Weyl chamber
\begin{align*}
\a^+=\{H\in\a\colon ~H_1>H_2>\dots>H_{p-1}>|H_p|\}.
\end{align*}

 The elements of the Weyl group $W$ act as permutations of the diagonal entries of $\mathcal{D}_X$ with 
eventual sign changes of any  even number of these
entries. The Lie algebra $\k$ is generated by the vectors $X_\alpha+ \theta X_\alpha$. We will use the notation
\begin{align*}
k^t_{X_\alpha}=e^{t(X_\alpha+ \theta X_\alpha)}.
\end{align*}
The linear space $\p$ has a basis formed by 
$A_i\in\a$, $ 1\leq i\leq p$ and by the symmetric matrices
$X_\alpha^s:= \frac12( X_\alpha-\theta X_\alpha)$ which have the
following form
\begin{align*}
Y_{i,j}:&= E_{i,p+j}+ E_{j,p+i}+E_{p+j,i}+ E_{p+i,j},\quad 1\leq i<j\leq p; \\
Z_{i,j}:&= - E_{i,p+j}+E_{j,p+i}-E_{p+j,i} +E_{p+i,j},\quad 1\leq i<j\leq p.
\end{align*}
Thus, if $X\in\p$ is as in (\ref{pp}), then  the vectors $A_i$ 
generate the diagonal entries of $B=X^s$ and   $Y_{i,j}$ and $Z_{i,j}$  the non-diagonal entries.

We now recall the useful matrix $S\in \SO(p+q)$ which allows us to
diagonalize simultaneously all the elements of $\a$. Let
\begin{align*}
S=\left[\begin{array}{cc}
\frac{\sqrt{2}}{2}\,I_p&\frac{\sqrt{2}}{2}\,J_p\\
\frac{\sqrt{2}}{2}\,I_p&-\frac{\sqrt{2}}{2}\,J_p
\end{array}
\right]
\end{align*}
where $J_p= (\delta_{i,p+1-i})$ is a matrix of size $p\times p$.
If $H=\left[\begin{array}{ccc}
0&\mathcal{D}_H\\
\mathcal{D}_H&0
\end{array}\right]$ with $\mathcal{D}_H=\diag[H_1,\dots,H_p]$ then 
$S^T\,H\,S=\diag[H_1,\dots,H_p,-H_p,\dots,-H_1]$.
The ``group'' version of this result is as follows:
\begin{align*}
S^T\,e^H\,S=
\diag[e^{H_1},\dots,e^{H_p},e^{-H_p},\dots,e^{-H_1}].
\end{align*}
%%%%%%%%%%%%%%%%%%%%%%%%%%%%%%%%%%%%
\begin{remark}\label{S}
The Cartan projection $a(g)$ on the group $\SO_0(p,p)$,
defined as usual by
\begin{align*}
g=k_1 e^{a(g)} k_2,\ \ \ a(g)\in \overline{\a^+}, k_1,\ k_2\in K
\end{align*}
is related to the singular values of $g\in \SO(p,p)$ in the following way.
Recall that the singular values of $g$ are defined as the non-negative
square roots
of the eigenvalues of $g^Tg$. Let us write $H=a(g)$. We have
\begin{align*}
g^T\,g &=
k_2^T\,e^{2\,H}\,k_2 
= (k_2^T\,S)\,(S^T\,e^{2\,H}\,S)\,(S^T\,k_2)
\end{align*}
where $S^T\,e^{2\,H}\,S$ is a diagonal matrix with nonzero entries
$e^{2\,H_1}$, \dots, $e^{2\,H_p}$,
$e^{-2\,H_p}$, \dots, $e^{-2\,H_1}$, satisfying 
$H_1\ge\ldots\ge H_{p-1}\ge  |H_p|$. Let us write $a_j=e^{H_j}$ for $j=1,\ldots,p-1$ and $a_p=e^{|H_p|}$.
Thus 
the set of $2p$ singular values of $g$ contains the values $a_1\geq \ldots \geq a_p\geq a_p^{-1}\geq \ldots \geq a_1^{-1} $
with $a_1\geq \ldots \geq a_p\geq 1$.

 Then
\begin{align*} 
a(g)=\left[\begin{array}{cc}
0&\mathcal{D}_{a(g)}\\
\mathcal{D}_{a(g)}&0\\
\end{array}\right]
\ \ \hbox{with}\ \ 
\mathcal{D}_{a(g)} =\diag[\log a_1,\dots,\log a_{p-1},{\rm sgn}(H_p)\log a_p].
\end{align*}

Note that this method does not allow us to distinguish between the situations where $H_p$ is positive or negative.
\end{remark}

{\bf Singular elements of $\a$.} In what follows, we will consider
singular elements $X$, $Y\in \partial\a^+$. We need to control the irregularity of $X$ and
$Y$, i.e. consider the simple positive roots annihilating $X$ and $Y$. A
special role is played by the last simple root $\alpha_{p}=H_{p-1}+H_p$,
different from the simple roots $\alpha_i(H)=H_i-H_{i+1}$, $i=1,\ldots,p-1$.  Note that $\alpha_p(X)=0$ implies that the last diagonal entry of $\mathcal{D}_X$ is negative or 0.

We introduce the following definition of the configuration of $X\in \overline{\a^+}$.
\begin{definition}\label{defConfig}
Let $X\in \overline{\a^+}$. 
In what follows, $x_i>0$, $x_i>x_j$ for $i<j$, $s_i\geq 1$, $u\geq 0$ and $\sum_{i=1}^r\,s_i +u	=p$. Let  ${\bf s}=(s_1,\ldots, s_r)$.
We define the configuration of $X$ by: 
\begin{align}
\ \ \ &[{\bf s};u] &{\it if}\ \  \mathcal{D}_X&=\diag\,[ \overbrace{x_{1},\dots,x_{1}}^{ {s}_{1}},
\overbrace{x_{2},\dots,x_{2}}^{ {s}_{2}},\dots,
\overbrace{x_{r},\dots,x_{r}}^{ {s}_{r}},
\overbrace {0,\dots, 0}^{ u}\ ] \label{ug0}  \\
%%%%%%%%%%%%%%%%%%%%%%%%%%%%%
\ \ \ &[{\bf s}] &{\it if}\ \  \mathcal{D}_X&=\diag\,[ \overbrace{x_{1},\dots,x_{1}}^{ {s}_{1}},
\overbrace{x_{2},\dots,x_{2}}^{ {s}_{2}},\dots,
\overbrace{x_{r-1},\dots,x_{r-1}}^{ {s}_{r-1}},
\overbrace {-x_r}^{s_r=1}\ ]\label{nil}  \\
%%%%%%%%%%%%%%%%%%%%%%%%%%%%%%%%%%%
\ \ \ &[{\bf s}]^-  &{\it if}\ \ \mathcal{D}_X&=\diag\,[ \overbrace{x_{1},\dots,x_{1}}^{ {s}_{1}},
\overbrace{x_{2},\dots,x_{2}}^{ {s}_{2}},\dots,
\overbrace{x_{r},\dots,x_{r},-x_r}^{s_r}].\label{prob}
\end{align}
We extend naturally the definition of configuration to any
$X\in\a$, whose configuration is defined as that of the projection
$\pi(X)$ of $X$ on $\overline{\a^+}$.
\end{definition}

\begin{remark}
We will often write $X=X[\ldots]$ when $[\ldots]$ is a configuration of $X$.
For the configuration (\ref{prob}), we write the $-$ superscript in  $X[{\bf s}]^-$ to indicate that  $X$  contains nonzero opposite entries.
We omit the  $-$ superscript in (\ref{nil}) and write $X[{\bf s}]$ because we are essentially in the same case as in (\ref{ug0}) with $u=0$
and $s_r=1$.

 If the number of zero entries $u=0$, it may be omitted in (\ref{ug0}). In particular, in the configurations (\ref{nil}) and (\ref{prob}), $u=0$.
Note that $X=0$ has configuration $[0;p]$. A regular $X\in \a^+$ has the configuration 
$[1^p;0]$ or $[1^{p-1};1]$, where $1^k=(1,\ldots,1)$ with $1$ repeated $k$ times.
\end{remark}

In what follows, we will write $\max {\bf s}=\max_i s_i$
and $\max({\bf s}, u)= \max(\max {\bf s},u)$.
We will show that in the case of the symmetric spaces
$\SO_0(p,p)/\SO(p)\times \SO(p)$, the criterion for
the existence of the density of the convolution
$\delta_{e^X}^\natural \star \delta_{e^Y}^\natural$ is given by the
following definition of an eligible pair $X$ and $Y$:

\begin{definition}\label{defEligible}
Let $X$ and $Y$ be two elements of $\a$
with configurations $[{\bf s};u]$ or $[{\bf s}]^-$    and $[{\bf t};v]$
or $[{\bf t}]^-$ respectively. We say
that $X$ and $Y$ are eligible if one of the two following cases holds 
\begin{align}
\hbox{$u\leq 1$, $v\leq 1$}\ \  and \ \ &\max({\bf s}) + \max({\bf t}) \le 2\,p-2,\label{cas1}\\
\hbox{$u\geq 2$ or $v\geq 2$}\ \  and \ \ &\max({\bf s}, 2u) + \max({\bf t}, 2v) \le 2\,p~~~~~~~~~~~~~~~\label{cas2}\\
\intertext{and, for $p=4,$}
 \{\mathcal{D}_X,\mathcal{D}_Y\}\not=\{\diag[a,a,a,a],\diag[b,b,c,c]\}\  nor \
&\{\diag[a,a,a,-a],\diag[b,b,c,-c]\}, \label{except4}
\end{align}
for any $a\not=0,~b\not=0,|b|\not=|c|$.
\end{definition}

Observe that if $X$ and $Y$ are eligible, then $X\not= 0$ and $Y\not= 0$. The  non-eligible pairs given by (\ref{except4}) are $\{ [4],[(2,2)]\}$,
$\{ [4]^-,[(2,2)]^-\}$ with $u=v=0$  or $\{ [4],[2;2]\}$ and $\{ [4]^-,[2;2]\}$ with $u=0$ and $v=2$.

%%%%%%%%%%%%%%%%%%%%%%%%%%%%%%%%%%%%%%%%%%%%%%%%%%%%%%
\begin{remark}
 It is interesting to note that the definition of eligible pairs is more complicated for the space $\SO (p,p)$ than for the spaces $\SO(p,q)$ with $p<q$ (recall that the latter spaces have a much richer root structure).  As for the spaces $\SO(p,q)$ with $p<q$, the number of zeroes on the diagonal of $\mathcal{D}_X$ is important.  Unlike in the case 
$\SO(p,q)$ with $p<q$, this only becomes a factor when the number of zeroes is greater than 1 (as opposed to greater than 0).

In \cite{PGPS_2013}, we showed that if $p<q$ and $X$ and $Y\in\a$ were such that the $\mathcal{D}_X$ and $\mathcal{D}_Y$ have no zero diagonal elements on the diagonal then $\mu_{X,Y}$ was absolutely continuous.  This is no longer the case when $p=q$  and this is one of  main differences between the $\SO (p,p)$ and $\SO(p,q)$ cases.  Another difference is the anomalous case when $p=4$ seen in (\ref{except4}) and the fact that lower dimensional cases require different proofs.  On the other hand,  when the number of zeroes on the diagonal of either $\mathcal{D}_X$ or $\mathcal{D}_Y$ is at least 2, then 
 the proof of Theorem  A is similar to the one found in \cite{PGPS_2013} but requires considering separately a low dimension case $X[5], Y[3;2]$.  
\end{remark}
%%%%%%%%%%%%%%%%%%%%%%%%%%%%%%%%%%%%%%%
%%%%%%%%%%%%%%%%%%%%%%%%%%%%%%%%%%%%%%%
\section{Basic tools and reductions}\label{tools}
%%%%%%%%%%%%%%%%%%%%%%%%%%%%%%%%%%%%%%%
%%%%%%%%%%%%%%%%%%%%%%%%%%%%%%%%%%%
\begin{definition}
For $Z\in\a$, let $V_Z$ be the subspace of $\p$ defined by 
\begin{align*}
V_Z=\hbox{span}\{ X_\alpha-\theta(X_\alpha)\ |\ \alpha(Z)\not=0\}\subset \p.
\end{align*}
We denote by $|V_Z|$ the dimension of $V_Z$. It equals the number of positive roots
$\alpha$ such that $\alpha(Z)\not=0$.
\end{definition}

The following definition and lemmas will help reduce the number of cases of configurations of $(X,Y)$ to verify.
\begin{definition}
We will say that $X$ and $X'\in\a$ are {\em relatives} if exactly one of the diagonal entries of $\mathcal{D}_X$ and $\mathcal{D}_{X'}$ differs by sign. 
If $X$ is a relative to $X'$ and $Y$ is a relative to $Y'$ then we will say that $(X,Y)$ is a relative pair of $(X',Y')$.  
\end{definition}

The properties listed in the following lemma are straightforward.
%%%%%%%%%%%%%%%%%%%%%%%%%%%%%%%%%%%%%%%%%%%%%%%%%%%%%%%
\begin{lemma}\label{relatives}

~\\ 

\begin{enumerate}
\item $X$ is a relative of $X$ if and only if $\mathcal{D}_X$ has at least one diagonal entry equal to 0, i.e. $u\ge 1$ in $X[{\bf s};u]$. Thus
 $X\in\overline{\a^+}$ has no other relatives in  $\overline{\a^+}$.
\item If $X$ and $X'$ are relatives then $|V_X|=|V_{X'}|$.
\item If $(X,Y)$ and $(X',Y')$ are relative pairs then $(X,Y)$ is an eligible pair if and only if $(X',Y')$ is an eligible pair.
\item If $X=X[{\bf s};u]\in\overline{\a^+}$ with $u>0$ then all $X_i$'s are non-negative.
\item If $X=X[{\bf s};u]$, $Y[{\bf t};v]\in\overline{\a}^+$ are such that $u>0$ or $v>0$ then either $\mathcal{D}_X$, $\mathcal{D}_Y$ have no negative entries or we can choose a relative pair $X'$, $Y'\in\overline{\a}^+$  in such a way that $\mathcal{D}_{X'}$, $\mathcal{D}_{Y'}$ have no negative entries. 
\item If $X$ is a relative of $X'$ and $X'$ is a relative of $X''$ then $X$ and $X''$ are in the same Weyl-group orbit.
\end{enumerate}
\end{lemma}

\begin{lemma}\label{cousin}
If $(X,Y)$ and $(X',Y')$ are relative pairs then either both sets $Ke^X\,K\,e^YK$ and $Ke^{X'}\,K\,e^{Y'}K$ have nonempty interiors or neither has.
\end{lemma}
\begin{proof}
Let $J_0=\diag\{\overbrace{1,\dots,1}^{2\,p-1},-1\}$ and note that $J_0$ is orthogonal and that $J_0^2\in\SO(p)\times \SO(p)$. 
Suppose that $X$, $Y$, $X'$ and $Y'$ are as in the statement of the lemma. Suppose that $w_1$, $w_2\in W\subset K$ are such that the 
element of $\mathcal{D}_{w_1\cdot X'}$ (resp. $\mathcal{D}_{w_2\cdot Y'}$) that differs by a sign from the corresponding element of $\mathcal{D}_{X}$ (resp. $\mathcal{D}_{Y}$)
is placed at the end. Then
\begin{align*}
K\,e^X\,K\,e^Y\,K
&=J_0\,K\,J_0\,w_1\,e^X\,w_1^{-1}\,J_0\,K\,J_0\,w_2\,e^Y\,w_2^{-1}\,J_0\,K\,J_0\\
&=J_0\,K\,(J_0\,w_1\,e^X\,w_1^{-1}\,J_0)\,K\,(J_0\,w_2\,e^Y\,w_2^{-1}\,J_0)\,K\,J_0
=J_0\,K\,e^{X'}\,K\,e^{Y'}\,K\,J_0
\end{align*}
which allows us to conclude.
\end{proof}

In the sequel we  use some ideas, results and notations of \cite[Section 3]{PGPS_Lie2010}, that we strengthen and complete.
\begin{prop}\label{KKK}
~\\

\begin{itemize}
\item[(i)] The density of the measure $m_{X,Y}$ exists if and only if its support
$Ke^X Ke^Y K$ has nonempty interior.
\item[(ii)] Consider the analytic map $T\colon K \times K \times K \to G$
defined by
\begin{align*}
T(k_1,k_2,k_3)=k_1\,e^X\,k_2\,e^Y\,k_3.
\end{align*}
The set $T(K\times K\times K)=Ke^X Ke^Y K$ contains an open set if and only if the derivative of $T$ is
surjective for some choice of ${\bf k}=(k_1,k_2,k_3)$.
\end{itemize}
\end{prop}

\begin{proof}
Part (i) follows from arguments explained in \cite{PGPS2} in the case
of the support of the measure $\mu_{X,Y}$, equal to $a(e^X\,Ke^Y).$
Part (ii) is justified for example in \cite[p. 479]{Helga}. 
\end{proof}

\begin{cor}\label{relativesReduction}
Let $(X,Y)$ and $(X',Y')$ be relative pairs.
The measure $m_{X,Y}$ is absolutely continuous if and only if
the measure $m_{X',Y'}$  is absolutely continuous.
\end{cor}
\begin{proof}
We use   Proposition  \ref{KKK}(i) and Lemma \ref{cousin}.
\end{proof}

\begin {prop}\label{U}
Let $U_Z=\k+\Ad(e^Z)\k$. The measure $m_{X,Y}$ is absolutely continuous if and only if there exists $k\in K$ such that 
\begin {align}\label{UXY}
U_{-X}+\Ad(k)U_Y=\g.
\end {align}
\end{prop}

\begin{proof}
We want to show that this condition is equivalent to the derivative of $T$ at ${\bf k}$ being surjective. We have 
\begin{align}
dT_{{\bf k}}(A,B,C)&=
\frac{d}{dt}\big|_{t=0}\ e^{tA}k_1\,e^X\,e^{tB}k_2\,e^Y\,e^{tC}k_3
\nonumber\\
&=A\,k_1\,e^X\,k_2\,e^Y\,k_3
+k_1\,e^X\,B\,k_2\,e^Y\,k_3
+k_1\,e^X\,k_2\,e^Y\,C\,k_3\label{A}
\end{align}
We now transform the space of all matrices of the form (\ref{A}) without modifying its dimension: 
\begin{align*}
&\dim\{A\,k_1\,e^X\,k_2\,e^Y\,k_3
+k_1\,e^X\,B\,k_2\,e^Y\,k_3
+k_1\,e^X\,k_2\,e^Y\,C\,k_3\colon A,B,C \in\k\}\\
=&\, \dim
\{k_1^{-1}\,A\,k_1\,e^X\,k_2\,e^Y
+e^X\,B\,k_2\,e^Y
+e^X\,k_2\,e^Y\,C\colon A,B,C \in\k\}\\
=&\, \dim
\{A\,e^X\,k_2\,e^Y
+e^X\,B\,k_2\,e^Y
+e^X\,k_2\,e^Y\,C\colon A,B,C \in\k\}\\
=&\, \dim
\{e^{-X}\,A\,e^X+B+k_2\,e^Y\,C\,e^{-Y}k_2^{-1}\colon A,B,C
\in\k\}
\end{align*}
The space in the last line equals 
$\k+\Ad(e^{-X})(\k)+\Ad(k_2)\,(\Ad(e^Y)(\k)) = U_{-X}+\Ad(k_2)U_Y$.
\end{proof}

In order to apply the condition (\ref{UXY}), we will consider convenient symmetrized  root vectors and the spaces $V_Z$
generated by them.

\begin{lemma}\label{3.2[7]}
Let $Z\in\a$. Then $U_Z=\k+V_Z=U_{-Z}$.
\end{lemma}
\begin{proof}
Clearly $V_Z=V_{-Z}$. We show that $V_Z\subset U_Z$ and therefore that $\k+V_Z\subset U_Z$: let $\alpha$ be a root such that $\alpha(Z)\not=0$. Note that
$[Z,X_\alpha]=\alpha(Z)\,X_\alpha$ and
$[Z,\theta(X_\alpha)]=-\alpha(Z)\,\theta(X_\alpha)$. Let $U=X_\alpha+\theta(X_\alpha)\in\k$. Now,
\begin{align*}
\Ad(e^Z)\, U
&=e^{\ad Z}\,(X_\alpha+\theta(X_\alpha))
=e^{\alpha(Z)}\,X_\alpha+e^{-\alpha(Z)}\,\theta(X_\alpha).
\end{align*}
Therefore $X_\alpha=(e^{\alpha(Z)}-e^{-\alpha(Z)})^{-1}\left(-e^{-\alpha(Z)}\, U
+ \Ad(e^Z)\, U\right)\in \k+\Ad(e^Z)(\k)=U_Z$. The vector $\theta X_\alpha$ is a root vector for the root $-\alpha$, so we also have $\theta X_\alpha\in U_Z$. 

It remains to show that $U_Z\subset\k+V_Z$. It suffices to show that $\Ad(e^Z)\,\k\subset\k+V_Z$: for every $\alpha$,
$\Ad(e^Z)\,(X_\alpha+\theta(X_\alpha))=e^{\alpha(Z)}\,X_\alpha+e^{-\alpha(Z)}\,\theta(X_\alpha)
=\frac{e^{\alpha(Z)}+e^{-\alpha(Z)}}{2}\,(X_\alpha+\theta(X_\alpha))+\frac{e^{\alpha(Z)}-e^{-\alpha(Z)}}{2}\,\,(X_\alpha-\theta(X_\alpha))
\in\k+V_Z$. 
\end{proof}

The following corollary is then straightforward:
\begin{cor}\label{VXVY}
The measure $m_{X,Y}$ is absolutely continuous if and only if  there exists $k\in K$ such that
\begin{align}\label{conditionSym}
V_X+\Ad(k)\,V_Y= \p.
\end{align}
\end{cor}

\begin{cor}\label{TOTAL}
The measure $m_{X,Y}$ is absolutely continuous if and only if  there exists a dense open subset $U\subset K$ such that for every $k\in U$
\begin{enumerate}
\item $V_{w_1 \cdot X}+\Ad(k)\,V_{w_2\cdot Y}= \p$ for every $w_1$, $w_2\in W$,
\item For every $r< 2p$, the matrix obtained by removing the first $r$ rows and $r$ columns of $k$ is non-singular.
\end{enumerate}
\end{cor}
\begin{proof}
First we note  that condition (\ref{conditionSym}) for some $k$ is actually equivalent to the existence of a dense open subset $U\subset K$ such that (\ref{conditionSym})
holds for every $k\in U$.
Indeed, since the equality (\ref{conditionSym}) can be expressed in terms of nonzero determinants, if it is satisfied for one value of $k$, it will be satisfied for every $k$ in a dense open subset of $K$. 

 In addition, (\ref{conditionSym}) is equivalent to the fact that $a(e^X\,K\,e^Y)$ has non-empty interior which, in turn, implies that $a(e^{w_1\cdot X}\,K\,e^{w_2\cdot  Y})$ has non-empty interior for any given $w_1$, $w_2\in W$ and for every 
$k\in U_{w_1,w_2}$ where $ U_{w_1,w_2}$ is open and dense.  Hence, for any given $w_1$, $w_2\in W$, there is a dense open set $U_{w_1,w_2}$ with 
$V_{w_1\cdot  X}+\Ad(k)\,V_{w_2\cdot  Y}= \p$.  For similar reasons, there exists a dense open subset of $K$ such that the second condition is satisfied (the condition being satisfied by the identity matrix).  Given that a finite intersection of dense open sets is a dense open set, the statement follows.
\end{proof}
\begin{remark}\label{order}
Corollary \ref{VXVY} and the fact that $V_{w\cdot X}=\Ad(w)\,V_X$ for $w\in W$ and $X\in \a$ ( \cite[Lemma 3.3]{PGPS_Lie2010}) imply that in the proof of
Theorem A one can assume that $X$ has a configuration $[{\bf s};u]$ with $s_1\geq s_2\geq \ldots \geq s_r$ or a configuration $[{\bf s}]^-$
with $s_1\geq s_2\geq \ldots \geq s_{r-1}$. The same remark applies to the configuration of $Y$.
\end{remark}

The following necessary criterion for the existence of the density will be very useful:
\begin{cor}\label{p2}
If $m_{X,Y}$ is absolutely continuous then $|V_X|+|V_Y|\geq \dim\p=p^2$. 
\end{cor}

The following definition and results will be helpful in resolving the exceptional case indicated in (\ref{except4}).

\begin{definition}
For $n\geq 1$, let $\Z(n)$ be the group formed by the matrices of the form 
\begin{align*}
\left[
\begin{array}{rrcrr}
\cos\theta_1&-\sin\theta_1\\
\sin\theta_1&\cos\theta_1\\
&&\ddots\\
&&&\cos\theta_r&-\sin\theta_r\\
&&&\sin\theta_r&\cos\theta_r
\end{array}
\right]
\end{align*}
where the last block is replaced by $1$ is $n$ is odd.
\end{definition}

\begin{remark}
Note that $\dim \Z(n)\leq n/2$ and that each element in $\Z(n)$ has a square root in $\Z(n)$.
\end{remark}

\begin{lemma}\label{Zn}
Consider $n\geq2$ and $k\in\SO(n)$. Then there exists $A\in\SO(n)$ such that 
$A^{-1}\,k\,A\in \Z(n)$.
\end{lemma}

\begin{proof}
Consult for example \cite{Friedberg}. Recall that the eigenvalues of $k$ are $e^{\pm i\theta_j}, \theta_j\in\R$.
\end{proof}

\begin{cor}\label{decompo}
Every matrix $\left[\begin{array}{cc}A_1&0\\0&A_2\end{array}\right]
\in \SO(p)\times \SO(p)$ can be written in the following format
\begin{align*}
\left[\begin{array}{cc}A_1&0\\0&A_2\end{array}\right]
&=\left[\begin{array}{cc}A&0\\0&A\end{array}\right]
\,\left[\begin{array}{cc}B&0\\0&B^{-1}\end{array}\right] \,\left[\begin{array}{cc}C&0\\0&C\end{array}\right] 
\end{align*}
with $A$, $C\in\SO(p)$ and $B\in \Z(p)$.
\end{cor}

\begin{proof}
According to Lemma \ref{Zn}, there exists a matrix $A\in\SO(p)$ such that
$B'=A^{-1}\,A_1\,A_2^{-1}\,A\in\Z(p)$. Pick $B\in\Z(p)$ such that
$B^2=B'$ and let $C=B^{-1}\,A^{-1}\,A_1$. Then $A\,B\,C=A_1$ and 
\begin{align*}
A\,B^{-1}\,C&=A\,B^{-1}\,(B^{-1}\,A^{-1}\,A_1)
=A\,(B^2)^{-1}\,A^{-1}\,A_1\\
&=A\,(B')^{-1}\,A^{-1}\,A_1
=A\,(A^{-1}\,A_1\,A_2^{-1}\,A)^{-1}\,A^{-1}\,A_1
=A_2
\end{align*}
which proves the lemma.
\end{proof}

\begin{remark}
The matrices $\left[\begin{array}{cc}B&0\\0&B^{-1}\end{array}\right]$ in the last corollary can be written as 
$\prod_{i=1}^{[p/2]}\,k^{t_i}_{Z^+_{2\,i-1,2\,i}}$ for an appropriate choice of $t_i$'s.
\end{remark}

In the proof of the necessity of the eligibility condition,
we will use the following result stated in \cite[Step 1, page 1767]{PGPS_Fun2010}.
Let the Cartan decomposition of $\SL(N,\F)$ be written as $g=k_1 e^{ \tilde
a(g)}k_2.$
\begin{lemma}\label{repeat}
Let $U=\diag([\overbrace{u_0,\dots,u_0}^r,u_1,\dots,u_{N-r}]$
and $V=\diag([\overbrace{v_0,\dots,v_0}^{N-s},v_1,\dots,v_s]$. where
$s+1\leq r<N$, $s\geq 1$, and the
$u_i$'s and $v_j$'s are arbitrary. 
Then each element of the diagonal of
$ \tilde
a(e^U\,\SU(N,\F)\,e^V)
$
has at least $r-s$ entries equal to $u_0+v_0$.
\end{lemma}
We will use Lemma \ref{repeat} with $N=p+q$ in the proofs of Proposition \ref{ness} and Theorem
\ref{*power}.\\

In the proof of Theorem A we will need the following technical
lemma from \cite{PGPS_2013}:
\begin{lemma}\label{calculs}~
\begin{enumerate}
\item For the root vectors $Z_{i,j}^+$ and $Y_{i,j}^+$, we have
\begin{align*}
\Ad(e^{t\,(Y_{i,j}^++\theta(Y_{i,j}^+))})(Y_{i,j}) &=\cos(4\,t)\,Y_{i,j}
+2\,\sin(4\,t)\,(A_i-A_j),\\
\Ad(e^{t\,(Z_{i,j}^+ +\theta(Z_{i,j}^+))})(Z_{i,j}) &=\cos(4\,t)\,Z_{i,j}
+2\,\sin(4\,t)\,(A_i+A_j).
\end{align*}
\item The operators 
$\Ad(e^{t\,(Y_{i,j}^++\theta(Y_{i,j}^+))})$ and $\Ad(e^{t\,(Z_{i,j}^++\theta(Z_{i,j}^+))})$ 
applied to the other symmetrized
root vectors do not produce any components in $\a$. 
\end{enumerate}
\end{lemma}
In the proof of the existence of the density for the pairs $X[4],Y[2,2]^-$ and $X[5], Y[3;2]$
without predecessors, we will need the following elementary lemma. Recall that $ Z_{k,l}=\frac12( Z_{k,l}^+ -\theta(Z_{k,l}^+) )\in \p$.
\begin{lemma}\label{CasesWithoutPred}
Let $1\le i <j\le p$ and  $1\le k <l\le p$. We have
\begin{align*}
[Z_{i,j}^+ +\theta(Z_{i,j}^+),\, Z_{k,l}]= & 
\left\lbrace
\begin{array}{ll}
0  &{\rm if}\ \{i,j\}\cap \{k,l\}=\emptyset,\\
4\,(A_i+A_j)   &{\rm if}\ \{i,j\}= \{k,l\},\\
2\,Y_{\min(j,l),\max(j,l)}   &{\rm if}\ i=k,\, j\not= l,\\
2\,Y_{\min(i,k),\max(i,k)}   &{\rm if}\ i\not=k,\, j= l,\\
-2\,Y_{i,l}     &{\rm if}\ j=k,\\
-2\,Y_{k,j}     &{\rm if}\ i=l.
\end{array}\right.
\end{align*}
\end{lemma}
\begin{proof}
 We apply the well known fact that $[\g_\alpha, \g_\beta]\subset \g_{\alpha+\beta}$
when $\alpha+\beta$ is a root and  $[\g_\alpha, \g_\beta]=0$ otherwise. For the computation of exact coefficients in the formulas,
we use Table 1.  
\end{proof}

%%%%%%%%%%%%%%%%%%%%%%%%%%%%%%%%%%%%%%%%%%%%%%%%%%%%%%%%%%%%%%%%%%%%%%%
\section{Necessity of the eligibility condition}\label{NSSE}
%%%%%%%%%%%%%%%%%%%%%%%%%%%%%%%%%%%%%%%%%%%%%%%%%%%%%%%%%%%%%%%%%%%%%%%
\begin{prop}\label{7cases}
If $X=X[{\bf s};u]$ and $Y=Y[{\bf t};v]\in\a$ with $u\leq 1$ and $v\leq 1$ are such that
$\max{\bf s}+\max{\bf t}>2\,p-2$ then $|V_X|+|V_Y|<p^2$.
\end{prop}

\begin{proof}
Assume that $X=X[{\bf s};u]$ and $Y=Y[{\bf t};v]\in\overline{\a}^+$. Without loss of generality, assume that $\max{\bf s}\geq \max{\bf t}$.  
We then have $\max {\bf s}=p$ and $\max{\bf t}\geq p-1$.  The only possible pairs are 
\begin{align}
&X[p], Y[p]\ \hbox{and the  relative pair}\  X'[p]^-,  Y'[p]^- \nonumber\\
&X[p], Y[p-1,1]\ \hbox{and the  relative pair}\  X'[p]^-,  Y'[1,p-1]^-\label{BAD}\\
&X[p], Y[p]^-\ \hbox{and the  relative pair }\ ( X', Y')= (Y,X)\nonumber\\\
&X[p], Y[1,p-1]^- \ \hbox{and the relative pair}\  X'[p]^-, Y'[p-1,1]. \nonumber\
\end{align}
By Remark \ref{order} we do not need to consider the configuration $[1,p-1]$.  We have $|V_X|=\frac{p(p-1)}{2}$ and $|V_{Y[p-1,1]}|=\frac{p(p-1)}{2} +p-1$.
We apply Lemma \ref{relatives} and we find by examination that $|V_X|+|V_Y|\leq p^2-1$ in all cases.
\end{proof}

\begin{cor}\label{uv} 
Let $p\ge 2$. 
Consider a pair $X=X[{\bf s};u]$ and $Y=Y[{\bf t};v]$ with $u\leq1$ and $v\leq1$. Then
$|V_X|+|V_Y|\ge p^2$ if and only if
\begin{align}\label{2p-2}
\max({\bf s}) + \max({\bf t}) \le 2\,p-2.
\end{align}
\end{cor}

\begin{proof}
By Proposition \ref{NSSE} only the sufficiency of condition (\ref{2p-2}) needs to be proven.
Suppose $\max({\bf s}) + \max({\bf t}) \leq 2\,p-2$ and that $\max{\bf s}\geq \max{\bf t}$.  If $\max {\bf s}=p$ then $|V_X|=p\,(p-1)/2$ and 
$\max{\bf t}\leq p-2$ implies $|V_Y|\geq p\,(p-1)/2+2(p-2)$.  If both $\max{\bf s}\leq p-1$ and $\max{\bf t}\leq p-1$ then 
$|V_X|\geq p\,(p-1)/2+p-1$ and $|V_Y|\geq p\,(p-1)/2+p-1$.  In both cases, the result follows.
\end{proof}

\begin{definition}
We will call {\bf exceptional} the set of configurations listed in (\ref{BAD})
and denote it by ${\mathcal E }$.
\end{definition}

\begin{prop}\label{espess}
Let $X$, $Y\in\a$ be such that $\mathcal{D}_X=\diag[b,b,c,c]$ with $b>c>0$ and $\mathcal{D}_Y=\diag[a,a,a,a]$ with $a>0$.  Then 
$\delta_{e^X}^\natural \star \delta_{e^Y}^\natural$ has no density.
\end{prop}

\begin{proof}
According to Corollary \ref{decompo}, we can write
\begin{align*}
K\,e^X\,K\,e^Y\,K=&K\,e^X\,
\left[\begin{array}{cc}A&0\\0&A\end{array}\right]
k^{t_1}_{Z^+_{1,2}}\,k^{t_2}_{Z^+_{3,4}}
\,\left[\begin{array}{cc}C&0\\0&C\end{array}\right]
\,e^Y\,K\\
=&K\,e^X\,
\,\exp(\left[\begin{array}{cccc}0&R&0&0\\-R&0&0&0\\0&0&0&R\\0&0&-R&0\end{array}\right])
k^{t_1}_{Z^+_{1,2}}\,k^{t_2}_{Z^+_{3,4}}\,
\,e^Y\,K
=K\,e^X
  k^{r_1}_{Y^+_{1,3}}\,k^{r_2}_{Y^+_{2,4}}
\,k^{t_1}_{Z^+_{1,2}}\,k^{t_2}_{Z^+_{3,4}}
\,e^Y\,K
\end{align*}
(here $Z_{1,2}^+$, $Z_{3,4}^+$, $Y_{1,3}^+$, $Y_{2,4}^+$ are exactly as in Table \ref{X}).
We used the fact that $e^Y$ and $\left[\begin{array}{cc}C&0\\0&C\end{array}\right]$ commute, the Cartan decomposition 
$
A=\left[\begin{array}{cc}A_1&0\\0&A_2\end{array}\right]
\,\exp(\left[\begin{array}{cc}0&R\\-R&0\end{array}\right])\,\left[\begin{array}{cc}C_1&0\\0&C_2\end{array}\right]
$
($R=\diag[r_1,r_2]$, $A_i$, $C_i\in\SO(2)$), and the facts that
$e^X$, $\left[\begin{array}{cccc}A_1&0&0&0\\0&A_2&0&0\\0&0&A_1&0\\0&0&0&A_2\end{array}\right]$  commute  
 and $\left[\begin{array}{cccc}C_1&0&0&0\\0&C_2&0&0\\0&0&C_1&0\\0&0&0&C_2\end{array}\right]$ commutes with 
$k^{t_1}_{Z^+_{1,2}}\,k^{t_2}_{Z^+_{3,4}}$ and with $e^Y$.

Now it is easy  to see by considering the proof of Proposition \ref{KKK}(ii) that for these particular $X$ and $Y$, the condition $V_X+\Ad(k)\,V_Y=\p$ must be satisfied by $k$ of the form 
$k_0= k^{r_1}_{Y^+_{1,3}}\,k^{r_2}_{Y^+_{2,4}}
\,k^{t_1}_{Z^+_{1,2}}\,k^{t_2}_{Z^+_{3,4}}$. 
On the other hand, $V_Y=\langle Z_{i,j},\ i<j\rangle$ and $\Ad(k_0)(Z_{i,j})$, $i<j$, can only produce diagonal elements which satisfy
$H_1+H_3=H_2+H_4$ as can be checked  by Lemma \ref{calculs} and using the fact that $\Ad(k_0)=\Ad(k^{r_1}_{Y^+_{1,3}})\,\Ad(k^{r_2}_{Y^+_{2,4}})
\,\Ad(k^{t_1}_{Z^+_{1,2}})\,\Ad(k^{t_2}_{Z^+_{3,4}}$).
Consequently, $\a\not\subset V_X+\Ad(k_0)\,V_Y$ and the density cannot exist.
\end{proof}

\begin{prop}\label{ness}
If $X$ and $Y$ are not eligible then the density does not exists.
\end{prop}

\begin{proof}
Let the configuration of $X$ be $[{\bf s}; u]$ or $[{\bf s}]^-$
and the configuration of $Y$ be $[{\bf t}; v]$ or $[{\bf t}]^-$.  Proposition \ref{7cases}, Proposition \ref{espess} and  Corollary 
\ref{relativesReduction} imply the non-existence of density when $X$ and $Y$ are not eligible and  $u\leq 1$ and  $v\leq 1$.   

Suppose then   that $u\geq 2$ or $v\geq 2$ and 
$\max({\bf s},2\,u)+\max({\bf t},2\,v)> 2\,p$ and consider the matrices $a(e^X\,k\,e^Y)$, $k\in \SO(p)\times \SO(p)$. Using (5) of Lemma \ref{relatives} and Lemma \ref{cousin}, we may assume that the diagonal entries of $\mathcal{D}_X$ and $\mathcal{D}_Y$ are non-negative.
Applying Remark \ref{S}, we have
\begin{align*}
\tilde a(e^X\,k\,e^Y)=\tilde a(
\overbrace{(S^T\,e^{X}\,S)}^{e^{S^T\,X\,S}}
\,\overbrace{(S^T\,k\,S)}^{\in\SO(p+q)}\,
\overbrace{(S^T\,e^{Y}\,S)}^{e^{S^T\,Y\,S}})
\end{align*}
where $\tilde a(g)$ is the diagonal matrix with the singular values of $g$ on  the diagonal, ordered decreasingly (see the explanation before Lemma 
\ref{repeat}).                                                                                                                                                                                                                                                                                   

If $u+v>p$ then there are 
$r-s=r+(N-s)-N=2\,u+2\,v-2\,p=2\,(u+v-p)$
repetitions of $0+0=0$ in coefficients of $ \tilde a(e^X k e^Y)$. Therefore 0 occurs at least $u+v-p>0$ times as a diagonal entry of 
$\mathcal{D}_H$ for every $H\in a(e^X\,K\,e^Y)$ which implies that $a(e^X\,K\,e^Y)$ has empty
interior.
If $2\,u+\max({\bf t})>2\,p$ denote $t=\max({\bf t})$. Let $Y_i\not=0$ be
repeated $t$ times in $\mathcal{D}_Y$ (or, if $t=t_r$ and $Y=Y[t]^-$, we have $t-1$ times $Y_r$ and once $-Y_r$ in $\mathcal{D}_Y$). Then there are 
$r-s=r+(N-s)-N=2\,u+t-2\,p$ repetitions of $Y_i+0$ in coefficients of $ \tilde a(e^X k e^Y)$. Therefore $Y_i$ occurs at 
least $2\,u+t-2p>0$ times as a diagonal entry of $\mathcal{D}_H$ for every $H\in a(e^X\,K\,e^Y)$ 
which implies that $a(e^X\,K\,e^Y)$ has empty interior. 
\end{proof}

%%%%%%%%%%%%%%%%%%%%%%%%%%%%%%%%%%%%%%%%%%%%%%%%%%%%%%%%%%%%%%%%%%
\section{Sufficiency of the eligibility condition}\label{Suff}
%%%%%%%%%%%%%%%%%%%%%%%%%%%%%%%%%%%%%%%%%%%%%%%%%%%%%%%%%%%%%%%%%%%%%

\subsection{Case $u\leq1$ and $v\leq1$}

\begin{remark}
In our proof, the case $u\leq1$ and $v\leq1$ is equivalent to the case $u=v=0$.  Indeed, for $H\in\overline{\a^+}$, if the sole diagonal entry 0 in $\mathcal{D}_H$ is replaced by a positive entry different from the existing diagonal entries of $\mathcal{D}_H$, then $V_H$ is unchanged. We will therefore assume in this section that $u=0$ and $v=0$.
\end{remark}

\begin{definition}
Let ${\bf s}=(s_1, s_2\dots, s_m)$ and ${\bf t}=(t_1, t_2, \dots, t_n)$ be two partitions of $p$ ($\sum_i\,s_i=p=\sum_i\,t_j$).  We will say that ${\bf s}$ is finer than ${\bf t}$
if the $t_i$'s are sums of disjoint subsets of the $s_j$'s (for example, ${\bf s}=[3,2,2,1,1,1]$ is finer than ${\bf t}=[5,3,2]=[3+2,2+1,1+1]$).  
\end{definition}

\begin{remark}
If $X=X[{\bf s}]$ and $Y=Y[{\bf t}]$ and ${\bf s}$ is finer than ${\bf t}$ then $V_Y\subset V_X$.
\end{remark}

In the following lemma, we reduce in a significant way the number of elements for which we must prove the existence of the density.
\begin{lemma}\label{FOND}
For $p\geq 5$, it is sufficient to prove the existence of the density in the following cases:
\begin{itemize}
\item[$S_1$:] $X[p], Y[p-k,k]$ for $p-k\ge k\ge 2$
 
\item[$S_2$:]  $X[p]^-, Y[p-k,k]$ for $p-k\ge k\ge 2$
 
\item[$S_3$:]  $X[1,p-1]^-, Y[p-1,1]$  
\item[$S_4$:]  $X[p-1,1], Y[p-1,1]$.
\end{itemize} 

For $p=4$ the same is true provided the case $S_1$ is replaced by the cases $X[4], Y[2,1,1]$ and $X[3,1], Y[2,2]$.
\end{lemma}

\begin{proof}
Suppose $p\geq 5$.
Let us call $A_0$ the configurations of the form $ [{\bf s}]$ and $A_1$ all the others, i.e. the configurations of the form $[{\bf s}]^-$.
 
\begin{itemize}

\item[(a)] We first observe that if the density exists for $S_1$ then it follows that it exists for all pairs $\{X,Y\}$ such that $X$, $Y\in A_0$,
except when $X$ or $Y$ have configurations $[p]$ or $[p-1,1]$. This comes from the fact that all these $X$, $Y$ have structures that are finer and, consequently, the corresponding $V_X$ and $V_Y$ are larger.  
Thus, existence of the density in the cases $S_1$ together with $S_4$ will imply the existence of the density in all the cases when $X$, $Y\in A_0$,
except when $\{X,Y\}\in {\mathcal E }$.
 
\item[(b)]  By switching to relatives and changing the order of $X$ and $Y$, we see that it implies the existence of the density in all the cases when $X$, $Y\in A_1$,
except when $\{X,Y\}\in {\mathcal E }$.
 
\item[(c)]  It remains to show that the cases $S_2$ and $S_3$ imply the existence of the density in all the cases when $X\in A_1$ and $Y\in A_0$,
except when $\{X,Y\}\in {\mathcal E }$. Note first that if $X=[{\bf s}]^-$ then  either $[{\bf s}]^-=[p]^-$ or $V_{X'}\subset V_X$ 
with $X'=X'[1,p-1]^-$.  In the first case, we observe that the case $S_2$ implies the cases $X[p]^-$ and
$Y\in A_0\setminus\{[p],[p-1,1]\}$ for the same reason as in (a). The only cases that remain with $Y[p-1,1]$ are covered by $S_3$.
Finally, switching to relatives, we get the pairs $X\in A_1, Y[p]$ which are not in  ${\mathcal E }$.
 
\end{itemize}
We illustrate the proof of the  Lemma in the case $p=5$.

\begin{align*}
\begin{array}{|c|c|c|c|c|c|c|c|c|c|c|c|c|c|c}
\multicolumn{1}{c}{[2,1^3]}&\multicolumn{1}{c}{[2,2,1]}
&\multicolumn{1}{c}{[3,1,1]}
&\multicolumn{1}{c}{[3,2]}&\multicolumn{1}{c}{[4,1]}
&\multicolumn{1}{c}{[5]}&
\multicolumn{1}{c}{[1^3,2]^-}
&\multicolumn{1}{c}{[2,1,2] ^-}
&\multicolumn{1}{c}{[1,1,3] ^-}
&\multicolumn{1}{c}{[3,2] ^-}
&\multicolumn{1}{c}{[2,3] ^-}
&\multicolumn{1}{c}{[1,4] ^-}
&\multicolumn{1}{c}{[5]^-}\\\cline{1-13}
\surd&\surd&\surd&\surd&\surd&\surd&\surd&\surd&\surd&\surd&\surd&\surd&\surd&\multicolumn{1}{c}{[2,1^3]}\\\cline{1-13}
\multicolumn{1}{c|}{}&\surd&\surd&\surd&\surd&\surd&\surd&\surd&\surd&\surd&\surd&\surd&\surd&\multicolumn{1}{c}{[2,2,1]}\\\cline{2-13}
\multicolumn{1}{c}{}&&\surd&\surd&\surd&\surd&\surd&\surd&\surd&\surd&\surd&\surd&\surd&\multicolumn{1}{c}{[3,1,1]}\\\cline{3-13}
\multicolumn{1}{c}{}&\multicolumn{1}{c}{}&&\surd&\surd&S_1&\surd&\surd&\surd&\surd&\surd&\surd&S_2&\multicolumn{1}{c}{[3,2]}\\\cline{4-13}
\multicolumn{1}{c}{}&\multicolumn{1}{c}{}&\multicolumn{1}{c}{}&&S_4&{\rm X}&\surd&\surd&\surd&\surd&\surd&S_3&{\rm X}&\multicolumn{1}{c}{[4,1]}\\\cline{5-13}
\multicolumn{1}{c}{}&\multicolumn{1}{c}{}&\multicolumn{1}{c}{}&\multicolumn{1}{c}{}&&{\rm X}&\surd&\surd&\surd&\surd&\surd&{\rm X}&{\rm X}&\multicolumn{1}{c}{[5]}\\\cline{6-13}
\multicolumn{1}{c}{}&\multicolumn{1}{c}{}&\multicolumn{1}{c}{}&\multicolumn{1}{c}{}&\multicolumn{1}{c}{}
&&\surd&\surd&\surd&\surd&\surd&\surd&\surd&\multicolumn{1}{c}{[1^3,2]^-}\\\cline{7-13}
\multicolumn{1}{c}{}&\multicolumn{1}{c}{}&\multicolumn{1}{c}{}&\multicolumn{1}{c}{}&\multicolumn{1}{c}{}&\multicolumn{1}{c}{}&&\surd&\surd&\surd&\surd&\surd&\surd&\multicolumn{1}{c}{[2,1,2]^-}\\\cline{8-13}
\multicolumn{1}{c}{}&\multicolumn{1}{c}{}&\multicolumn{1}{c}{}&\multicolumn{1}{c}{}&\multicolumn{1}{c}{}&\multicolumn{1}{c}{}&\multicolumn{1}{c}{}&&\surd&\surd&\surd&\surd&\surd&\multicolumn{1}{c}{[1,1,3]^-}\\\cline{9-13}
\multicolumn{1}{c}{}&\multicolumn{1}{c}{}&\multicolumn{1}{c}{}&\multicolumn{1}{c}{}&\multicolumn{1}{c}{}&\multicolumn{1}{c}{}&\multicolumn{1}{c}{}&\multicolumn{1}{c}{}&&\surd&\surd&\surd&\surd&\multicolumn{1}{c}{[3,2 ]^-}\\\cline{10-13}
\multicolumn{1}{c}{}&\multicolumn{1}{c}{}&\multicolumn{1}{c}{}&\multicolumn{1}{c}{}&\multicolumn{1}{c}{}&\multicolumn{1}{c}{}&\multicolumn{1}{c}{}&\multicolumn{1}{c}{}&\multicolumn{1}{c}{}&&\surd&\surd&\surd&\multicolumn{1}{c}{[2,3]^-}\\\cline{11-13}
\multicolumn{1}{c}{}&\multicolumn{1}{c}{}&\multicolumn{1}{c}{}&\multicolumn{1}{c}{}&\multicolumn{1}{c}{}&\multicolumn{1}{c}{}&\multicolumn{1}{c}{}&\multicolumn{1}{c}{}&\multicolumn{1}{c}{}&\multicolumn{1}{c}{}&&\surd&{\rm X}&\multicolumn{1}{c}{[1,4]^-}\\\cline{12-13}
\multicolumn{1}{c}{}&\multicolumn{1}{c}{}&\multicolumn{1}{c}{}&\multicolumn{1}{c}{}&\multicolumn{1}{c}{}&\multicolumn{1}{c}{}&\multicolumn{1}{c}{}&\multicolumn{1}{c}{}&\multicolumn{1}{c}{}&\multicolumn{1}{c}{}&\multicolumn{1}{c}{}&\multicolumn{1}{c|}{}&{\rm X}&\multicolumn{1}{c}{[5]^-}\\\cline{13-13}
\end{array}
\end{align*}

In the above table, $\surd$ indicates that the pair is eligible,
X indicates that the pair is not eligible and therefore that the density does not exist (the cases identified in (\ref{BAD})) and the $S_i$'s correspond to the notation above (the pair is eligible where the $S_i$'s appear).
We use the reduction from Remark \ref{order}.
\end{proof}
 
\begin{theorem}\label{existenceNOZEROS}
Let $p\ge 2$ and suppose that $X=X[{\bf s};u]$ or $[{\bf s}]^-$ and $Y=Y[{\bf t};v]$ or  $[{\bf t}]^-$  are such that $u\leq1$ and $v\leq1$. If the pair $\{X,Y\}$ does not belong to the set ${\mathcal E }\cup
\{[4],[2,2]\}\cup\{[4]^-,[2,2]^-\}$ then the density exists.
\end{theorem}

\begin{proof} 
The proof proceeds by induction similarly as in \cite{PGPS_2013}, but with a different ``asymmetric'' technique of
executing Steps 2 and 3. Also, for small values of $p$ the proofs must be led separately, due the lack of available good predecessors.
With some exceptions in the starting phase of the induction, and in the case $S_3$ of the Lemma \ref{FOND}, the elements $X$ and $Y$ will be in $\overline{\a^+}$
and their ``usual'' predecessors will be obtained by skipping the first diagonal terms of $\mathcal{D}_X$ and $\mathcal{D}_Y$.

The only case of existence of the density for $p=2$ is for regular $X[1,1]$ and $Y[1,1]$ (according to (\ref{BAD}), only
the pair $X[1,1],Y[1,1]$ of two regular elements verifies $|V_X|+|V_Y|\geq p^2$).

For $p=3$, we have four possible configurations of nonzero singular elements:
$[2,1]$, $[3]$, $[1,2]^-$, $[3]^-$.

All exceptional cases listed in (\ref{BAD}) appear and only three pairs of singular configurations, namely 
($X[2,1]$, $Y[2,1]$), ($X[2,1]$, $Y[1,2]^-$) and ($X[1,2]^-$, $Y[1,2]^-$),
verify $|V_X|+|V_Y|\ge p^2$.  Given that the pairs 	$(X[2,1], Y[2,1])$ and $(X[1,2]^-, Y[1,2]^-)$ are relatives, we only have to check the cases 
$(X[2,1], Y[2,1])$ and $(X[2,1],Y[1,2]^-)$.

In the case  $X[2,1]$, $Y[2,1]$, we write $\mathcal{D}_X=\diag[a,a,b]$, $\mathcal{D}_Y=\diag[c,c,d]$ and the predecessors
$\mathcal{D}_{X'}=\diag[a,b]$, $\mathcal{D}_{Y'}=\diag[c,d]$, obtained by skipping the first coordinates, are regular.  
In the case $X[2,1]$, $Y[1,2]^-$ we  consider 
$\mathcal{D}_X=\diag[a,a,b]$, $\mathcal{D}{w\cdot y}=\diag[-d,c,d]$ and only now go to regular predecessors $\mathcal{D}_{X'}=\diag[a,b]$, $\mathcal{D}_{(w\cdot Y)'}=\diag[c,d]$. The general proof given below applies in these cases.

When $p=4$, by Lemma \ref{FOND}, we must show the existence of the density for:

\begin{enumerate} 
\item $X[4], Y[2,1,1]$ and $X[2,2], Y[3,1]$,
 
\item $X[4]^-, Y[2,2]$ or equivalently the relative pair $X[4],Y[2,2]^-$, 
 
\item $X[1,3]^-, Y[3,1]$, 
 
\item $X[3,1], Y[3,1]$.
\end{enumerate}

In the cases (1), (3) and (4), the usual predecessors have density when $p=3$. The general proof given below applies in these cases. 
 
For the case (2), observe that when $p=3$, the configuration $X'[3]$ never gives the existence of density, when $Y'$ is singular.
That is why the second case $X[4], Y[2,2]^-$ has no good predecessors and  this case must be proved separately.
We will do it after the general proof.
 
Starting from $p=5$, the {\bf general proof} by induction applies, the exceptions due to small values of $p$ being taken care of. We present this proof now.
 
{\bf Step 1.} Let $Y\in \overline {\a^+}$ be such that $\mathcal{D}_Y=\diag[\overbrace{a,\dots,a}^{p-k},\overbrace{b,\dots,b}^k]$ and let its predecessor $Y'$ be such that $\mathcal{D}_{Y'}=\diag[\overbrace{a,\dots,a}^{p-k-1},\overbrace{b,\dots,b}^k]$.   The space $V_Y$ is generated by completing a basis of $V_{Y'}$ with
\begin{align*}
N_Y=\{ Y_{1,p-k+1},\ldots, Y_{1,p}, Z_{1,2},\ldots, Z_{1,p} \}.
\end{align*} 

We choose the predecessor $X'$ of $X$ in the same manner, except in the case $S_3$, where we first write $\mathcal{D}_X=\diag[b,b,\dots,b,a,-b]$ where $a>b>0$ and skip the first term $b$ in $\mathcal{D}_{X}$.

It is easy to see that the predecessors of $X$ and $Y$ are in the corresponding classes $S_i$ for $p-1$, so with density, except
for $X[5]$,  $Y[3,2]$, due to the non-eligible case $X[4], Y[2,2]$. In this last case we arrange $\mathcal{D}_X=\diag[a,a,a,a,a]$ and $\mathcal{D}_{wY} =\diag[-b,b,b,c,-c]$ and go down to good predecessors $X'[4]$, ($wY)'[2,2]^-$.
The proof described below leads to the existence of the density.
 
By the induction hypothesis and considering Corollary \ref{TOTAL}, there exists an open dense subset $D'$ of $\SO(p-1)\times\SO(p-1)$ such that for all $w'\in W'$ and $k_0\in D'$,
\begin{align}\label{VX'VY'p'}
V_{w'\cdot X'} +\Ad(k_0) V_{Y'}=\p'
\end{align}
and $k_0$ verifies  condition (2) of Corollary \ref{TOTAL}.

We embed $K'= \SO(p-1)\times\SO(p-1)$ in $\SO(p)\times\SO(p)$ in the following manner:
\begin{align*}
K'\ni k'=
\left[ 
\begin{array}{cccc}
1& & & \\
& k_{0,1} &&\\
&&1& \\
&&& k_{0,2}
\end{array} \right] \in
\left[ 
\begin{array}{cc}
\SO(p) &\\
& \SO(p)
\end{array} \right],\ \ k_{0,1},k_{0,2}\in \SO(p-1).
\end{align*}

Hence, we have (identifying  $\p'$ with its natural embedding  into $\p$)
\begin{align}
V_1:= V_{w'\cdot X'} +\Ad(k_0) V_{Y'}=\p'=\left[ \begin{array}{cc}
0 & B'\\
B'^T& 0
\end{array}
\right]\label{p}
\end{align}
for any $w'\in W'$, where $B'=\left [\begin{array}{c|c}0_{1\times 1}&0_{1\times (p-1)}\\ \hline
0_{(p-1)\times 1}&B''_{(p-1)\times( p-1)}
\end{array}\right]$, and the matrix $B''$ is arbitrary (note that $\p'$ is of dimension $(p-1)^2$).
We must show that for some $k\in K$, the space $V_X+\Ad(k)\,V_Y= \p$.
 
{\bf Step 2.} The element $Y$ is always of the same form, so the next step of the proof is common for all the 4 cases. 
Similarly as in \cite{PGPS_2013}, Step 2 of the proof of Theorem 4.8, case (i), we prove that
for $k_0\in  D'\subset \SO(p-1)\times\SO(p-1)$ the following property holds:

The space $\Ad(k_0)\hbox{span}(N_Y)$ is of dimension
 $p+k-1$ and its elements can be written in the form 
\begin{align*}
\left[\begin{array}{c|ccc}0&
a_1&\dots&a_{p-1}\\ \hline
\tau_1& \\
\vdots&\\
\tau_{p-1-k}&&0\\
a_{p+k-1}&\\
\vdots\\
a_{p}&
\end{array}\right]^s
\end{align*}
with $a_i\in\R$ arbitrary and
$\tau_j=\tau_j(a_1,\dots,a_{p+k-1})$.
We will not need to write explicitly the functions $\tau_j$. 
For the sake of completeness, we give a proof of Step 2.

Step 2 comes from the fact that the action of $\Ad(k_0)$ on the elements of $N_Y$ gives the linearly
independent matrices
\begin{align}
\Ad(k_0)Y_{1,i}=\left [\begin{array}{c|c}0& \beta_{i-1}^T\\ \hline
\alpha_{i-1}& 0
\end{array}\right]^s,\ i=p-k+1,\dots, p,\quad 
Ad(k_0)Z_{1,i}=\left
[\begin{array}{c|c}0& -\beta_{i-1}^T\\ \hline
\alpha_{i-1}&0
\end{array}\right]^s, i=2,\dots, p\label{with}
\end{align}
where the $\alpha_i$'s are the columns of $k_{0,1}$ and the $\beta_i$'s are the columns of $k_{0,2}$. 
Let us write $\alpha_i'$ for a column $\alpha_i$ with the
first $p-1-k$ entries omitted. In order to prove the statement of
Step 2, we must show that the matrices obtained by replacing
$\alpha_i$ by $\alpha_i'$ in (\ref{with})
are still linearly independent. This is equivalent to the linear independence of the matrices
\begin{align}
\left [\begin{array}{c|c}
0& -\beta_i^T\\ \hline
\alpha_i'& 0\\
\end{array}\right]^s,~ i=1,\dots, p-k-1,~~
\left [\begin{array}{c|c}0& \beta_i^T\\ \hline
0& 0\\
\end{array}\right]^s,~ i=p-k,\dots, p-1,~~
\left [\begin{array}{c|c}0&0\\ \hline
\alpha_i'& 0\\
\end{array}\right]^s~ i=p-k,\dots, p-1.
\label{prime}
\end{align}
The matrices in (\ref{prime}) are linearly independent given that the matrix $k_0$ was assumed to satisfy condition (2) of Corollary \ref{TOTAL}.

Observe that contrary to \cite{PGPS_2013}, we have filled the zero margins of the matrix $B'$ asymetrically,
which is why we call this method ``asymmetric''. The reason for doing this will be clear from the structure of the set $N_X$
that we study now.
 
{\bf Step 3.} Let us write the set $N_X$ in the four cases $S_i$:
\begin{enumerate}
\item $S_1$: $N_X= \{Z_{1,2},\ldots, Z_{1,p} \}$
 
\item $S_2$: $N_X= \{Z_{1,2},\ldots, Z_{1,p-1},Y_{1,p} \}$
 
\item $S_3$: $N_X= \{Z_{1,2},\ldots, Z_{1,p-1}, Y_{1,p-1}, Y_{1,p} \}$
 
\item $S_4$: $N_X= \{Z_{1,2},\ldots, Z_{1,p},Y_{1,p} \}$.
\end{enumerate} 
	
We will now use the elements of $N_X$ in order to generate thw missing $p-k-1$ dimensions $\tau_j$ in the margins of $B'$.
We use for this the vectors $Z_{1,2}$, \ldots, $Z_{1,p-k}$ available in all four cases for $k\ge 1$. We proceed as follows:
 
If $\tau_1(1,0,\ldots,0)=-1$, the vector $Z_{1,2}\in N_X$ is unhelpful. We change $X'$ into $X''$ by putting the sign $-$ before the second
and the last term of $X'$. We obtain $X''=w''\cdot X'$ such that $N_X$ contains $Y_{1,2}$ instead of $Z_{1,2}$ and $w''\in W'$ changes two signs of 
$X'$. This manipulation is justified
by the fact that (\ref{VX'VY'p'}) holds for any $w''\in W'$. We repeat this procedure, if needed, whenever $\tau_j( {\bf e}_j)=-1$
and obtain, from elements of $N_{w\cdot X}$ and $\Ad(k_0)\,(N_Y)$,
\begin{align}
\left[\begin{array}{c|ccc}0&
a_1&\dots&a_{p-1}\\ \hline
a_{2\,p-2}& \\
\vdots&\\
a_{p+k-1}&&0\\
\vdots\\
a_p&
\end{array}\right]^s\label{ais}
\end{align}
for $w\cdot X$ with $w\in W$ and where the $a_i$'s are arbitrary.

{\bf Step 4.}   Noting that we have at least one element of $N_X$ that has not been used, either $Y_{1,p}$ or $Z_{1,p}$, combining (\ref{VX'VY'p'}) and (\ref{ais}), we have 
\begin{align}
V_0:=\widetilde{V_{w\cdot X}}+\Ad(k_0)(V_Y)=\left[ 
\begin{array}{c|ccc}
0&*&\dots&*\\ \hline
*&*&\dots&*\\
\vdots&\vdots&\ddots&\vdots\\
*&*&\dots&*
\end{array}
\right]\label{V0}
 \end{align}
where $\widetilde{V_{w\cdot X}}$ corresponds to the all of $V_X$ without using the remaining  $Y_{1,p}$ or $Z_{1,p}$.  To fix things, let us assume that the 
unused element is $Z_{1,p}$, the reasoning being similar if it is $Y_{1,p}$ instead.

The end of the proof is similar to the final step of the proof in \cite{PGPS_2013}. 
Refer to Lemma \ref{calculs} and note that for $t$ small, $\Ad(e^{t\,(Z_{i,j}^+ +\theta\,Z_{i,j}^+)})(\widetilde{V_{w\cdot X}})+\Ad(k_0)(V_Y)=V_0$.  Indeed, the lemma shows that no new element is introduced and, for $t$ small, the dimension is unchanged.  On the other hand, $V_0$ is strictly included in $\Ad(e^{t\,(Z_{i,j}^+ +\theta\,Z_{i,j}^+)})(\langle Z_{1,p}\rangle \cup\widetilde{V_{w\cdot X}})+\Ad(k_0)(V_Y)$ since, still by Lemma \ref{calculs}, a new diagonal element is introduced.  We conclude that for $t$ small enough, $\Ad(e^{t\,(Z_{i,j}^+ +\theta\,Z_{i,j}^+)})(\langle Z_{1,p}\rangle \cup\widetilde{V_{w\cdot X}})+\Ad(k_0)(V_Y)=\p$.  Finally,
\begin{align*}
V_{w\cdot X}+\Ad(e^{-t\,(Z_{i,j}^+ +\theta\,Z_{i,j}^+)}\,k_0)(V_Y)=\p.
\end{align*}

{\bf The case $ X[2,2]^-$, $Y[4]$:} 
This case is awkward since $(X,Y)$ do not have eligible predecessors.  We select $X$ and $ Y$ such that $\mathcal{D}_X=\diag[a,a,b,-b]$ and 
$\mathcal{D}_{ Y}=\diag[c,c,c,c]$ (assuming $a$, $b$, $c\not=0$ and $a\not=b$).  
Then $V_Y$ is generated by the basis $B_Y$ composed of  all the 6 vectors $Z_{i,j}$, $i<j$,  while  the basis $B_X$ of $V_X$ contains the vectors $Y_{1,3}$, $Y_{1,4}$, $Y_{2,3}$, $Y_{2,4}$, $Y_{3,4}$ and all the vectors $Z_{i,j}$ except $Z_{3,4}$.  Note that $|V_X|=10$.

For a root vector $Z_{i,j}^+$ denote $Z_{i,j}^\k=Z_{i,j}^+ +\theta(Z_{i,j}^+)\in\k$. Define 
$Z_0=Z_{1,2}^\k+ Z_{2,3}^\k+Z_{1,4}^\k+Z_{2,4}^\k\in\k$. We denote 
\begin{align*}
F_t&=V_X +\Ad(e^{tZ_0})V_Y=V_X + e^{t\ad Z_0}(V_Y), \\
E_t&=V_X +\langle \{ v + t[Z_0,v]:\ v\in V_Y\} \rangle.
\end{align*}
We will write
\begin{align*}
f_t=\det(B_X,\Ad(e^{tZ_0})B_Y)
\end{align*}
where the elements of $\p$ are seen as column vectors in $\R^{p^2}$.
Analogously, we denote by $e_t$ the determinant constructed in a similar way
from the vectors of $B_X$ and the vectors $ v + t[Z_0,v]$, $v\in B_Y$, belonging to $E_t$. We write $f_t=e_t + r_t$ and we analyse now $e_t$ and $r_t$ in order to show that $f_t\not=0$ for some small nonzero $t$.

Using  Lemma \ref{CasesWithoutPred}, we check  that $e_t=c t^5$ with $c=\det(B_X, Z_{3,4},[Z_0,Z_{1,2}],\ldots, [Z_0,Z_{2,4}])\not =0$.
The coefficient of $t^6$ in $e_t$ equals zero since $\det(B_X,[Z_0,B_Y])=0$. On the other hand it is easy to see in a similar way that the remainder $r_t$ in the analytic expansion
$f_t=e_t + r_t$ does not have terms in $t^n$ for $n<6$. We conclude that $f_t\not=0$   for small nonzero $t$.
\end{proof}

\subsection{Case $u\geq 2$ or $v\geq2$}	

This case is handled in much of the same way as the case $u>0$ or $v>0$ in \cite{PGPS_2013}.  A notable difference is that the basis for induction is the previous case ($u\leq 1$ and $v\leq 1$). For $p=3$, we only need to consider the pair $X[1;2]$, $Y[2,1]$ which has regular predecessors. Similarly, for $p=4$, we see that all  eligible pairs with
$u\geq 2$ or $v\geq 2$ have eligible predecessors. 

In the case $p=5$, because of (\ref{except4}), there are eligible pairs with no eligible predecessors. It suffices to consider  the pair 
$X[5]$, $Y[3;2]$. In order to show that the density exists in this case, we use the same technique as for the case  $ X[2,2]^-$, 
$Y[4]$. We take $Z_0=Z_{1,2}^\k + Z_{2,3}^\k+ Z_{3,4}^\k+ Z_{1,5}^\k+ Z_{2,5}^\k $.
In order to prove that $e_t=c t^9$ with $c\not=0$, we check using Lemma \ref{CasesWithoutPred} that the 9 vectors $[Z_0,Z_{1,2}],\ldots, [Z_0,Z_{3,5}]$ produce the missing vectors $Y_{1,2}$, $Y_{1,3}$, $Y_{2,3}$, $Y_{4,5}$ and the diagonal.

Starting from $p=6$, the induction proof works. The fact that
 the roots $\alpha_i$ defined by $\alpha_i(X)=X_i$ are absent in the case $\SO(p,p)$ does not influence the proof (in \cite{PGPS_2013}  the roots  $\alpha_i$
were not used in Step 4 of the proof).
These two differences being overcome, the proof is sufficiently similar that it should not be repeated here.
%%%%%%%%%%%%%%%%%%%%%%%%%%%%%%%%%%%%%%%%%%%%%%%%%%%%%%%%%
\section{Applications}\label{app}
We now extend our results to the symmetric spaces of type $C_p$, i.e. to the complex and quaternion cases.
Recall that $\SU(p,p)$ is the subgroup of ${\bf SL}(2p,\C)$ such that $g^*\,I_{p,p}\,g=I_{p,p}$ 
while $\Sp(p,p)$ is the subgroup of ${\bf SL}(2p,\H)$ such that $g^*\,I_{p,p}\,g=I_{p,p}$. Their respective maximal compact subgroups are ${\bf S}({\bf U}(p)\times {\bf U}(p))$ and 
$\Sp(p)\times \Sp(p)\equiv \SU(p,\H)\times \SU(p,\H)$.
Their subspaces $\p$ can be described as $\left[\begin{array}{cc}0&B\\ B^*&0\end{array}\right]$ where 
$B$ is an arbitrary complex (respectively quaternionic) matrix of size 
$p\times p$. The Cartan subalgebra $\a$ is chosen in the same way as for $\so(p,p)$, with real entries in the diagonal.

\begin{remark}\label{complex}
The following table is helpful in showing the differences and similarities between $\SO(p,p)/\SO(p)\times \SO(p)$, $\SU(p,p)/{\bf S}({\bf U}(p)\times{\bf U}(p))$ and $\Sp(p,p)/\Sp(p)\times \Sp(p)$ (the ``real'', ``complex'' and ``quaternionic'' cases). 
\begin{align*}
\begin{array}{|l|c|c|c|c|}\hline
&{\scriptstyle \SO(p,p)/\SO(p)\times \SO(p)}&{\scriptstyle \SU(p,p)/{\bf S}({\bf U}(p)\times{\bf U}(p))}
&{\scriptstyle \Sp(p,p)/\Sp(p)\times \Sp(p)}\\ \hline
\hbox{Root system}&D_p&C_p&C_p\\\hline
\parbox[t]{2.5cm}{$m_\alpha$ where\\ $\alpha(X)\\=X_i-X_j$, $i<j$}&1&2&4\\ \hline
\parbox[t]{2.5cm}{$m_\alpha$ where\\ $\alpha(X)\\=X_i+X_j$, $i<j$}&1&2&4\\ \hline
\parbox[t]{2.5cm}{$m_\alpha$ where\\ $\alpha(X)=2\,X_i$,\\ $i=1$, \dots, $p$}&0&1&3\\ \hline
\hbox{Dimension of $\p$}&p^2&2\,p^2&4\,p^2\\\hline
\hbox{\parbox[t]{2.5cm}{Action of the Weyl group on $X\in \a$}}
&\hbox{\parbox[t]{2.5cm}{Permutes the diagonal entries of $\mathcal{D}(X)$ and changes any pair of signs}}
&\hbox{\parbox[t]{2.5cm}{Permutes the diagonal entries of $\mathcal{D}(X)$ and changes any sign}}
&\hbox{\parbox[t]{2.5cm}{Permutes the diagonal entries of $\mathcal{D}(X)$ and changes any sign}}
\\\hline
\end{array}
\end{align*}
\end{remark}

\begin{theorem}\label{iff}
Consider the symmetric spaces 
$\SU(p,p)/{\bf S}({\bf U}(p)\times{\bf U}(p))$ and $\Sp(p,p)/{\bf
Sp}(p)\times\Sp(p)$.
Let $X=X[{\bf s};u]$ and $Y=Y[{\bf t};v]\in\a$. Then the measure $\delta_{e^X}^\natural \star
\delta_{e^Y}^\natural$ is absolutely continuous if and only if $\max({\bf s}, 2u) + \max({\bf t}, 2v) \le 2\,p$.
\end{theorem}
\begin{proof}
Let $X$, $Y\in\a$. Note that since 
\begin{align*}
 a(e^X\,{\bf S}({\bf U}(p)\times{\bf U}(p))\,e^Y)\subset a(e^X\,(\Sp(p)\times\Sp(p))\,e^Y),
\end{align*}
if the density exists in the complex case, it also exists in the quaternionic case. 
On the other hand, given Lemma \ref{repeat}, one can reproduce Proposition \ref{ness} using $\F=\C$ and $\F=\H$ to show that the condition is necessary in 
the complex and quaternionic cases. 

However, the root structure is richer in the complex and quaternionic cases compared to the real cases.  The existence of the roots $\alpha(X)=2\,X_k$ makes the  complex and quaternionic cases very similar to the case $q>p$.  

It clearly suffices to prove the result in the complex case.  The involution $\theta$ is given by $\theta(X)=-X^*$ and the positive root vectors are generated 
by $X_k^+=\left[\begin{array}{c|c}-i\,E_{k,k}&i\,E_{k,k}\\\hline-i\,E_{k,k}&i\,E_{k,k}\end{array}\right]$ for the root $\alpha(H)=2\,H_k$, 
by $Y_{r,s}^+=\left[\begin{array}{c|c}E_{r,s}-E_{s,r}&E_{r,s}+E_{s,r}\\\hline E_{r,s}+E_{s,r}&E_{r,s}-E_{s,r}\end{array}\right]$,
$Y_{r,s,{\bf C}}^+=\left[\begin{array}{c|c}i\,(E_{r,s}+E_{s,r})&i\,(E_{r,s}-E_{s,r})\\\hline i\,(E_{r,s}-E_{s,r})&i\,(E_{r,s}+E_{s,r})\end{array}\right]$for the root $\alpha(H)=H_r-H_s$ and by
$Z_{r,s}^+=\left[\begin{array}{c|c}E_{r,s}-E_{s,r}&E_{s,r}-E_{r,s}\\\hline E_{r,s}-E_{s,r}&E_{s,r}-E_{r,s}\end{array}\right]$, 
$Z_{r,s,{\bf C}}^+=\left[\begin{array}{c|c}-i\,(E_{r,s}+E_{s,r})&i\,(E_{r,s}+E_{s,r})\\\hline-i\,(E_{r,s}-E_{s,r})&i\,(E_{r,s}+E_{s,r})\end{array}\right]$ 
for the root $\alpha(H)=H_r+H_s$ (here the matrices $E_{r,s}$ are of size $p\times p$).

Taking into account the fact that if the density exists in the real case, it also exists in the complex case, we only have a few cases to verify.  Given that changing any sign of a diagonal element of $\mathcal{D}_X$, $X\in\a$, is a Weyl group action, we can always assume that all entries of $\mathcal{D}_X$ are non-negative.  The configuration $[s]^-$ thus disappears. 

We will need to show that the cases $(X[p],Y[p])$, $(X[p],Y[p-1;1])$ and $(X[4],Y[2;2])$ all have a density.  We will use the case $p=1$ which is of rank 1 as the inductive step (there is nothing to prove for that case).

For the case $X[p]$, $Y[p]$, $p>1$, we proceed much as in \cite{PGPS_2013}.  We note here a few differences.  If $k_0\in {\bf S}({\bf U}(p)\times{\bf U}(p))$ with 
$k_0=
\left[ 
\begin{array}{cccc}
1& & & \\
& k_{0,1} &&\\
&&1& \\
&&& k_{0,2}
\end{array} \right] \in
{\bf S}({\bf U}(p)\times{\bf U}(p))$ then 
\begin{align}
\Ad(k_0)\,(Z_{1,k})&=\left [\begin{array}{c|c}
0& -\beta_{k-1}^*\\ \hline 
\alpha_{k-1}&0
\end{array}\right]^s,
\Ad(k_0)\,(Z_{1,k,{\bf C}})=\left [\begin{array}{c|c}
0& i\,\beta_{k-1}^*\\ \hline 
i\,\alpha_{k-1}&0
\end{array}\right]^s,
~k=1,\dots,p-1\ \ \hbox{and}~\label{S1}\\
\Ad(k_0)\,(X_{1})&=\left [\begin{array}{c|c}
i& 0\\ \hline 
0&0
\end{array}\right]^s.\nonumber
\end{align}

Given that 
\begin{align}
Z_{1,j}=\left [\begin{array}{c|c}
0& -{\bf e}_{j-1}^T\\ \hline 
{\bf e}_{j-1}&0
\end{array}\right]^s, Z_{1,j,{\bf C}}=\left [\begin{array}{c|c}
0& i\,{\bf e}_{j-1}^T\\ \hline 
i\,{\bf e}_{j-1}&0
\end{array}\right]^s,\quad ~j=2,\dots, p,\label{Z1}
\end{align}
we want to show that the matrices in (\ref{S1}) together with
those of (\ref{Z1}) are linearly independent for a $k_0\in{\bf S}({\bf U}(p-1)\times{\bf U}(p-1))$
for which the equality (\ref{VX'VY'p'}) holds. 
Note that if $k_{0,1}=i\,I_{p-1}$, $k_{0,2}=-i\,I_{p-1}$  then the matrices in (\ref{S1}) and (\ref{Z1}) are linearly independent. Since the linear independence is based on a determinant being nonzero, this implies
that the set of matrices $k_0$ for which this is true is open and dense in ${\bf S}({\bf U}(p-1)\times{\bf U}(p-1))$.  
We conclude that if $N_X'=N_X\backslash \{X_1\}$ then ${\rm span}\,(N_X' + V_{X'}) + \Ad(k'_0) V_Y$ has the form
$\scriptstyle \left[ 
\begin{array}{c|ccc}
i\,a&*&\dots&*\\ \hline
*&*&\dots&*\\
\vdots&\vdots&\ddots&\vdots\\
*&*&\dots&*
\end{array}
\right]$ where the $*$'s represent arbitrary complex numbers and $a$ is an arbitrary real number. 

We finish the proof as in the case $\SO(p,p)$ using the vector $X_1^+$. 

The case $(X[p],Y[p-1;1])$ has eligible predecessors  $(X'[p-1],Y'[p-1])$.  We then have $N_X=\{Z_{1,k},Z_{1,k,{\bf C}}, X_1\}$ and 
$N_Y=\{Z_{1,k},Z_{1,k,{\bf C}}, Y_{1,k},Y_{1,k,{\bf C}}\}$.  The rest follows easily.

Finally, the case $(X[4],Y[2;2])$ has predecessors $(X'[3],Y'[2;1])$ which are eligible.
\end{proof}
We will conclude this paper with two further applications.
\begin{prop}
Let $X$ and $Y\in\a$ be such that
$\left(\delta_{e^X}^\natural\right)^{*2}$ and 
$\left(\delta_{e^Y}^\natural\right)^{*2}$ are absolutely continuous. Then 
$\delta_{e^X}^\natural*\delta_{e^Y}^\natural$ is absolutely continuous.
\end{prop}
\begin{proof}
The proof is very similar to the one found in \cite{PGPS_2013}.
\end{proof}

In previous papers, we have studied a related question: if $X\in\a$ and $X\not=0$, for which convolution powers $l$
is the measure $\left(\delta_{e^X}^\natural\right)^{l}$ absolutely continuous?
This problem is equivalent to the study of the absolute continuity of convolution powers of uniform orbital measures $\delta_{g}^\natural=m_K*\delta_{g}*m_K$
for $g\not\in K$.
It was proved in \cite[Corollary 7]{PGPS_Fun2010} that it is always the case for $l\ge r+1$, where
$r$ is the rank of the symmetric space $G/K$. It was also shown in \cite{PGPS_2013}
that $r+1$ is optimal for this property for symmetric spaces of type $A_n$ (\cite[Corollary 18]{PGPS_Fun2010})
but this is not the case for the symmetric spaces of type $B_p$ where $r$ was shown to be sufficient in \cite{PGPS_2013}.

\begin{prop}
If $p=3$ and $\mathcal{D}_X=\diag[a,a,a]$, $a>0$, then $(\delta_{e^X}^\natural)^3$ is not absolutely continuous in \\
 $\SO(p,p)/\SO(p)\times \SO(p)$ while $(\delta_{e^X}^\natural)^4$ is absolutely continuous.
\end{prop} 

\begin{proof}
Computing the derivative of the map $T(k_1,k_2,k_3,k_4)=k_1\,e^X\,k_2\,e^X\,k_3\,e^X\,k_4$ at $(k_1,k_2,k_3,k_4)$ 
as in (\ref{A}), we obtain 
\begin{align*}
k_1\,e^X\,k_2\,(\Ad(k_2^{-1})\,U_{-X}+\Ad(e^X)\,U_X)\,e^X\,k_3\,k_4=
k_1\,e^X\,k_2\,(\k+\Ad(k_2^{-1})\,V_{-X}+\Ad(e^X)\,V_X)\,e^X\,k_3\,e^X\,k_4.
\end{align*}

The dimension of this space is at most $|\k|+|V_{-X}|+|V_X|=|\k|+3+3<|\k|+|\p|=|\g|$ so the map cannot be surjective.

On the other hand, $X'$ such $\mathcal{D}_{X'}=\diag[2\,a,a,a]$ belongs to $a(e^X\,K\,e^X)$ from what precedes (taking $x=a$) and the pair $(X',X')$ is eligible. 
From there, we conclude that $a(e^X\,K\,e^X\,K\,e^X\,K\,e^X)$ has nonempty interior.
\end{proof}

\begin{prop}
If $p=4$ and $\mathcal{D}_X=\diag[a,a,a,a]$, $a>0$, then $(\delta_{e^X}^\natural)^3$ is not absolutely continuous in $\SO(p,p)/\SO(p)\times \SO(p)$ while $(\delta_{e^X}^\natural)^4$ is absolutely continuous.  Consequently, if $X=X[{\bf s};u]$ with $u\geq 1$ or $X=X[{\bf s}]^-$ then 
$(\delta_{e^X}^\natural)^4$ is absolutely continuous.
\end{prop}

\begin{proof}
We know that the elements of $e^X\,K\,e^X$ have the form $k_a\,e^Z\,k_b$ where $\mathcal{D}_Z=\diag[c,c,d,d]$, $c\geq d$.  From the end of the proof of Proposition \ref{espess}, we know that for all $H\in  a(e^X\,K\,e^X\,K\,e^X)$,  $\mathcal{D}_H=\diag[H_1,H_2,H_3,H_4]$ will satisfy 
$H_1+H_3=H_2+H_4$.  We conclude therefore that $a(e^X\,K\,e^X\,K\,e^X)$ has empty interior.  On the other hand, since there exists $Z\in a(e^X\,K\,e^X)$ with $\mathcal{D}_Z=\diag[c,c,d,d]$ with $c>d>0$ and $(Z,Z)$ forms an eligible pair, it follows that $a(e^X\,K\,e^X\,K\,e^X\,K\,e^X)$ has nonempty interior since it contains $a(e^Z\,K\,e^Z)$. 
\end{proof}

\begin{prop}
If $p\geq 5$ and $\mathcal{D}_X=\diag[a,\dots, a]$, $a>0$, then $(\delta_{e^X}^\natural)^3$ is absolutely continuous in $\SO(p,p)/\SO(p)\times \SO(p)$. Consequently, if $X=X[{\bf s};u]$ with $u\geq 1$ or $X=X[{\bf s}]^-$ then $(\delta_{e^X}^\natural)^3$ is absolutely continuous.
\end{prop} 

\begin{proof}
Note that for $t>0$ small enough, $Z=a(e^X\,k_t^{Z_{p-1,1}^+}\,e^X)\in a(e^X\,K\,e^X)$ is such that $\mathcal{D}_Z=\diag[\overbrace{a,\dots,a}^{p-2},x,x]$ with $a>x>0$.  Given that $(Z,X)$ form an eligible pair and that $a(e^Z\,K\,e^X)\subset a(e^X\,K\,e^X\,K\,e^X)$, the result follows.
\end{proof}

\begin{theorem}\label{*power}
On symmetric spaces $\SO_0(p,p)/\SO(p)\times\SO(p)$, ($p\geq 4$), $\SU(p,p)/{\bf S}({\bf
U}(p)\times{\bf U}(p))$ and $\Sp(p,p)/\Sp(p)\times\Sp(p)$, $p\geq 2$,
for every nonzero $X\in\a$, the measure $(\delta_{e^X}^\natural)^p$ is absolutely
continuous. Moreover, $p$ is the smallest value for which this is
true: if $X$ has a configuration $[1;p-1]$ then the measure $(\delta_{e^X}^\natural)^{p-1}$
is singular.
\end{theorem}
\begin{proof}
We used the three preceding propositions and \cite[Theorem 5.3]{PGPS_2013}.
\end{proof}

\section{Conclusion}
With this paper and with \cite{PGPS_Lie2010, PGPS_2013}, we have now obtained sharp criteria on singular $X$ and $Y$ for the existence of the density of $\delta_{e^X}^\natural \star
\delta_{e^Y}^\natural$ for the root systems of types $A_n$, $B_p$, $C_p$, $D_p$ and $E_6$. Thanks to \cite{PGPS_Fun2010, PGPS_2013} and Theorem \ref{*power} of the present paper, sharp criteria
are now given for the $l$-th convolution powers $ (\delta_{e^X}^\natural)^l$ to be absolutely continuous for any $X\not=0$, $X\in\a$. 
It is interesting to note that the eligibility criterion  depends strongly on the geometry of the root system. Consequently, 
 a characterization of eligibilty  that would be applicable for all Riemannian symmetric 
spaces of non-compact type is unlikely to exist. 

\end{document}